\documentclass[12pt]{amsart}

\usepackage{amsmath,amssymb,amsthm,enumerate}
\usepackage{ytableau}
\usepackage[margin=1in]{geometry}

\numberwithin{equation}{section}

\DeclareMathOperator{\RE}{Re}
\DeclareMathOperator{\IM}{Im}

\newcommand{\sm}{\setminus}

\newcommand{\N}{\mathbb{N}}
\newcommand{\Z}{\mathbb{Z}}

\newcommand{\R}{\mathbb{R}}

\newcommand{\E}{\mathbb{E}}

\newcommand{\ve}{\varepsilon}

\newcommand{\1}{^{-1}}

\newcommand{\f}[2]{\frac{#1}{#2}}

\newcommand{\Mod}[1]{\ (\mathrm{mod}\ #1)}

\newtheorem{theorem}{Theorem}[section]

\newtheorem{lemma}[theorem]{Lemma}

\numberwithin{theorem}{section}

\title{On even entries in the character table of the symmetric group}
\author{Sarah Peluse}
\address{Mathematical Institute, University of Oxford, Radcliffe Observatory Quarter, Woodstock Road, Oxford OX2 6GG, United Kingdom}
\email{sarah.peluse@maths.ox.ac.uk}

\begin{document}
\begin{abstract}
We show that almost every entry in the character table of $S_n$ is even as $n\to\infty$. This resolves a conjecture of Miller. We similarly prove that almost every entry in the character table of $S_n$ is zero modulo $3,5,7,11,$ and $13$ as $n\to\infty$, partially addressing another conjecture of Miller.
\end{abstract}

\maketitle

\section{Introduction}\label{sec1}

In~\cite{Miller2019}, Miller conjectured that the proportion of odd entries in the character table of $S_n$ goes to zero as $n$ goes to infinity, based on computational data for $n$ up to $76$. It has been known for a long time, due to work of McKay~\cite{McKay1972}, that only a vanishingly small proportion of characters of the symmetric group have odd degree. Recently, Gluck~\cite{Gluck2019} showed that, more generally, the density of odd entries tends to zero in columns of the character table corresponding to cycle types whose odd factor has few distinct cycle lengths. This is still a very sparse set of columns, however. Morotti~\cite{Morotti2020} has also shown that, for any fixed prime $p$, a different sparse set of columns consists almost completely of multiples of $p$.

In this paper, we prove Miller's conjecture.
\begin{theorem}\label{thm1.1}
The density of odd entries in the character table of $S_n$ goes to $0$ as $n\to\infty$.
\end{theorem}
We will present two proofs of this result. The first is shorter and specific to the prime $2$. The second is longer, but generalizes to some other primes, and is quantitatively stronger. While the conjecture resolved by Theorem~\ref{thm1.1} is the main focus of~\cite{Miller2019}, Miller also presents data for the number of entries in the character table of $S_n$ (for $n$ up to $38$) that are divisible by $3$ and $5$. Based on this, he conjectured that, for any fixed prime $p$, the proportion of the character table of $S_n$ that is not divisible by $p$ tends to zero as $n$ goes to infinity. We also verify this conjecture for $p\leq 13$.
\begin{theorem}\label{thm1.2}
Let $p=2,3,5,7,11,$ or $13$. There exists $\delta_p>0$ such that the density of entries of the character table of $S_n$ that are not divisible by $p$ is at most $O_p(n^{-\delta_p})$.
\end{theorem}
In contrast to the situation modulo primes, it appears from the data in~\cite{Miller2019} that the proportion of zeros in the character table of $S_n$ may tend to a positive constant. Our methods do not say anything on this matter, though very recently Larsen and Miller~\cite{LarsenMiller2020} have shown that the proportion of nonzero entries of character tables of groups of Lie type goes to zero as the rank goes to infinity.

The remainder of this paper is organized as follows. In Section~\ref{sec2}, we recall some useful facts about characters of the symmetric group and set up the proofs of Theorems~\ref{thm1.1} and~\ref{thm1.2}. We prove Theorem~\ref{thm1.1} in Section~\ref{sec3} using a short argument specific to the prime $2$, and prove the more general Theorem~\ref{thm1.2} in Section~\ref{sec4}.

\section*{Acknowledgments}
The author thanks Persi Diaconis for bringing this problem to her attention and David Gluck and Alex Miller for helpful comments on earlier drafts of this paper.

The author is supported by the NSF Mathematical Sciences Postdoctoral Research Fellowship Program under Grant No. DMS-1903038

\section{Preliminaries and setting up the arguments}\label{sec2}

For any two partitions $\lambda$ and $\mu$ of $n$, let $\chi_{\mu}^{\lambda}$ denote the value of the character of $S_n$ corresponding to the partition $\lambda$ on the congruence class of permutations with cycle type $\mu$. The proofs of Theorems~\ref{thm1.1} and~\ref{thm1.2} will require a couple of basic results about characters of the symmetric group. The first gives us a useful condition for determining character values modulo primes (see, for example, Section~3 of~\cite{OdlyzkoRains2000} or Proposition~1 of~\cite{Miller2019}).
\begin{lemma}\label{lem2.1}
  Let $\lambda$ and $\mu$ be partitions of $n$. If $\nu$ is another partition of $n$ that can be obtained from $\mu$ by replacing $p$ parts of the same size $m$ by one part of size $pm$, then we have
  \[
\chi_{\mu}^{\lambda}\equiv\chi_{\nu}^\lambda\Mod{p}.
  \]
\end{lemma}
To state the next result, we will need to recall some concepts from the study of partitions. For any box $u$ in the Young diagram of a partition $\lambda$, the \textit{hook-length} of $u$ is the number of boxes directly to the right of $u$ plus the number of boxes directly below $u$ plus $1$. To illustrate, the Young diagram of $\lambda=(5,3,2,1,1)$ below has each box labeled with its hook-length.
\begin{center}
	\begin{figure}[h]
		\ytableausetup{mathmode}
		\begin{ytableau}
		9 & 6 & 4 & 2 & 1 \\
		6 & 3 & 1 \\
                4 & 1 \\
		2 \\
                1 \\
		\end{ytableau}
		\caption{Hook-lengths for $\lambda=(5,3,2,1,1)$.}
                \label{fig1}
	\end{figure}
\end{center}

For any positive integer $t$, we say that a partition is a \textit{$t$-core} if its Young diagram contains no hooks of length divisible by $t$. From Figure~\ref{fig1} we see that $(5,3,2,1,1)$ is a $5$-core. Because a Young diagram has a border strip of length $t$ if and only if it has $t$ as a hook-length, the following result is an immediate consequence of the Murnaghan--Nakayama rule (see also the proof of Theorem~1 of~\cite{Miller2019}).
\begin{lemma}\label{lem2.2}
Let $\lambda$ and $\mu$ be partitions of $n$. If $\mu$ has a part of size $t$ and $\lambda$ is a $t$-core, then $\chi_{\mu}^\lambda=0$.
\end{lemma}
Finally, we will also need the fact that most partitions of $n$ are $t$-cores whenever $t$ is substantially larger than $\f{\sqrt{6}}{2\pi}\sqrt{n}\log{n}$, which happens to be the expected size of the largest part of a random partition of $n$.

\begin{lemma}\label{lem2.3}
  Suppose that $t=\f{\sqrt{6}}{2\pi}\sqrt{n}\log{n}+\f{\sqrt{6}}{\pi}M\sqrt{n}$ for some $0\leq M\leq 10\log{n}$. Then there are at least $p(n)(1-O(\f{\log{n}}{e^{M}}))$ $t$-core partitions of $n$.
\end{lemma}
\begin{proof}
  From Lemma~5 of~\cite{Morotti2020}, the number of $t$-core partitions of $n$ is at least $p(n)-(t+1)p(n-t)$. By taking the first term in Rademacher's formula~\cite{Rademacher1937} for $p(n)$, one gets the explicit estimate
\[
  p(n)=\f{1+O(n^{-1/2})}{4\sqrt{3}n}e^{\pi\sqrt{2n/3}
  }.
\]
Thus,
\begin{align*}
  1-(t+1)\f{p(n-t)}{p(n)}&= 1-O\left(\sqrt{n}\log{n}\cdot\exp\left(\pi\sqrt{\f{2}{3}}\cdot\left(\sqrt{n-t}-\sqrt{n}\right)\right)\right) \\
  &= 1-O\left(\sqrt{n}\log{n}\cdot\exp\left(\pi\sqrt{\f{2n}{3}}\cdot\left(\sqrt{1-\f{t}{n}}-1\right)\right)\right) \\
                          &= 1-O\left(\sqrt{n}\log{n}\cdot\exp\left(-\f{\pi}{\sqrt{6}}\cdot\f{t}{\sqrt{n}}\right)\right),
\end{align*}
which equals $1-O(\f{\log{n}}{e^M})$ by recalling the expression for $t$.
\end{proof}
For context, Erd\H{o}s and Lehner~\cite{ErdosLehner1941} showed that a random partition of $n$ has a part of size at least $\f{\sqrt{6}}{2\pi}\sqrt{n}\log{n}+\f{\sqrt{6}}{\pi}\log{\f{\sqrt{6}}{\pi}}+\f{\sqrt{6}}{\pi}M\sqrt{n}$ with probability $1-e^{-e^{-M}}$, so the largest part $\mu_1$ of a typical partition $\mu$ of $n$ will not be large enough for Lemma~\ref{lem2.3} to give a nontrivial lower bound on the number of partitions of $n$ are $\mu_1$-cores. Lemmas~\ref{lem2.2} and~\ref{lem2.3} taken together only guarantee that the character table of $S_n$ contains a $\Omega(\f{1}{\log{n}})$-proportion of zeros.

Now fix a prime $p$. For any partition $\mu$ of $n$, let $\tilde{\mu}$ denote the partition of $n$ obtained by repeatedly replacing any $p$ parts of the same size $m$ in $\mu$ with one part of size $pm$ until there are no parts appearing more than $p-1$ times. For example, if $p=2$ and $\mu=(6,2,1,1,1)$, then $\tilde{\mu}=(6,4,1)$. By Lemma~\ref{lem2.1}, if $\chi^\lambda_{\tilde{\mu}}=0$, then $\chi_{\mu}^\lambda$ is a multiple of $p$. Our strategy will thus be to show that $\chi_{\tilde{\mu}}^\lambda=0$ for almost all pairs of partitions $\lambda$ and $\mu$ of $n$. We will do this by showing that the partition $\tilde{\mu}$ typically has largest part $\tilde{\mu}_1$ significantly larger than the largest part of a random partition when $p\leq 13$ -- so large that almost all partitions of $n$ are $\tilde{\mu}_1$-cores. Theorems~\ref{thm1.1} and~\ref{thm1.2} will then follow from Lemmas~\ref{lem2.1},~\ref{lem2.2}, and~\ref{lem2.3}.

To prove that $\tilde{\mu}_1$ is typically large, we study the following partition statistic. For any positive integer $k$ relatively prime to $p$ and partition $\mu$ of $n$, set
\[
M_{\mu}^{(k)}:=\sum_{\substack{\mu_i\in\mu \\ \mu_i=kp^j}}\mu_i.
\]
So, $M_{\mu}^{(k)}$ is the sum of all parts of $\mu$ of the form $k$ times a power of $p$. For example, when $p=2$ we have $M_{(6,2,1,1,1)}^{(1)}=5=M_{(6,4,1)}^{(1)}$. Note that not only does $M^{(k)}_{\mu}=M^{(k)}_{\tilde{\mu}}$ for all $k$ and partitions $\mu$ more generally, but also that the parts of the form $kp^j$ appearing in $\tilde{\mu}$ can be read off from the base-$p$ expansion of $\f{1}{k}M_{\tilde{\mu}}^{(k)}$. That is, the base-$p$ expansion of $\f{1}{k}M_{\mu}^{(k)}$ has an $i$ in the $p^j$ place if and only if $\tilde{\mu}$ has $i$ parts of size $kp^j$. It thus suffices to show that there is a $\gamma_p>0$ such that, for almost every $\mu$, there exists some $k$ for which
\begin{equation}\label{eq2.1}
\left\lfloor \log_p{\f{M_\mu^{(k)}}{k}}\right\rfloor\geq \log_p\f{(1+\gamma_p)\f{\sqrt{6}}{2\pi}\sqrt{n}\log{n}}{k},
\end{equation}
since this will guarantee that $\tilde{\mu}$ has a part of size at least $(1+\gamma_p)\f{\sqrt{6}}{2\pi}\sqrt{n}\log{n}$, which is easily large enough for Lemma~\ref{lem2.3} to give us a strong bound on the number of partitions that are not $\tilde{\mu}_1$-cores.

For each fixed $k$, the quantity $M_{\mu}^{(k)}$ turns out to have average size approximately equal to $\f{\sqrt{6}}{2\pi\log{p}}\sqrt{n}\log{n}$. Note that $\f{1}{\log{2}}>1$, while $\f{1}{\log{p}}<1$ whenever $p>2$. Thus, to prove Theorem~\ref{thm1.1}, it will suffice to show that $M^{(k)}_{\mu}$ is typically close to its average for enough values of $k$ that~\eqref{eq2.1} must hold for one of them. It turns out that proving this for $k=1,3,$ and $5$ is enough, so we will just have to compute the first and second moments of $M^{(1)}_{\mu},M^{(3)}_{\mu},$ and $M^{(5)}_{\mu}$ and invoke Chebyshev's inequality. Section~\ref{sec3} contains the details of this argument.

Since $M_\mu^{(k)}$ is typically small for each individual $k$ when $p>2$, we will need to consider many $k$'s simultaneously to handle larger primes. It turns out that, for each fixed $k\leq n^{1/2-\ve_1(p)}$, the quantity $M_\mu^{(k)}$ has size $<(1+\ve_2(p))\f{\sqrt{6}}{2\pi}\sqrt{n}\log{n}$ with probability at most $1-\f{1}{n^{1/2-\ve_3(p)}}$ for some $\ve_1(p),\ve_2(p)$, and $\ve_3(p)$ satisfying $\ve_3(p)>\ve_1(p)$. Thus, if the $M_\mu^{(k)}$ were all independent for $k$ up to $n^{1/2-\ve_1(p)}$, we would be able to prove Theorem~\ref{thm1.2} for all primes $p$. However, we can only prove independence when $\ve_1(p)>1/4$, and it turns out that $\ve_1(p)>1/4$ if and only if $p\leq 13$, which explains the restriction to these primes in Theorem~\ref{thm1.2}. Section~\ref{sec4} contains the details of this argument.

\section{Proof of Theorem~\ref{thm1.1}}\label{sec3}
Fix $p=2$ for this section and let $\sum_{\mu\vdash n}$ and $\E_{\mu\vdash n}$ denote the sum and average, respectively, over partitions $\mu$ of $n$. In the following lemma, we compute the first two moments of $M_\mu^{(k)}$.
\begin{lemma}\label{lem3.1}
  For any odd $k$, we have
  \[
\E_{\mu\vdash n}M_\mu^{(k)}=\f{\sqrt{6}}{2\pi\log{2}}\sqrt{n}\log{n}(1+o_k(1))
  \]
  and
  \[
\E_{\mu\vdash n}(M_{\mu}^{(k)})^2=\left(\f{\sqrt{6}}{2\pi\log{2}}\sqrt{n}\log{n}\right)^2(1+o_k(1)).
  \]
\end{lemma}
\begin{proof}
  Let $P(x):=\sum_{n=0}^\infty p(n)x^n$ denote the generating function of the number of partitions of $n$ and
\[
  H_k(x,u):=\prod_{j=0}^{\infty}\f{1}{1-(ux)^{k2^j}}\prod_{\ell\neq k2^j}\f{1}{1-x^\ell}=\sum_{n=1}^\infty\sum_{m=0}^\infty h(n,m)x^nu^{m}
\]
denote the bivariate generating function of the number of partitions $\mu$ of $n$ with $M_{\mu}^{(k)}=m$. By logarithmically differentiating $H_k(x,u)$ with respect to $u$ and then setting $u=1$, one sees that the generating function for $\sum_{\mu\vdash n}M^{(k)}_{\mu}$ is
  \[
P(x)\sum_{j=0}^{\infty}\f{k2^jx^{k2^j}}{1-x^{k2^j}}=:P(x)F_k(x).
\]
Similarly, by logarithmically differentiating $H_k(x,u)$ twice with respect to $u$ and then setting $u=1$, one also sees that the generating function for $\sum_{\mu\vdash n}(M^{(k)}_{\mu})^2$ is
  \[
P(x)\left(F_k(x)^2+\sum_{j=0}^\infty\f{k^24^jx^{k2^j}}{(1-x^{k2^j})^2}\right)=:P(x)G_k(x).
\]

Asymptotics for $\sum_{\mu\vdash n}M^{(k)}_{\mu}$ and $\sum_{\mu\vdash n}(M^{(k)}_{\mu})^2$ can be gotten by a standard saddle-point argument. The argument is so standard, in fact, that the sort of saddle-point analysis that we would need to do has already been carried out in general by Grabner, Knopfmacher, and Wagner~\cite{GrabnerKnopfmacherWagner2014} for any quantity whose generating function takes the form $P(x)H(x)$ with $H$ sufficiently well-behaved:
\begin{theorem}[Theorem~2.2 of~\cite{GrabnerKnopfmacherWagner2014}]\label{thm3.2}
  Suppose that $H$ is a function on the complex disk satisfying
  \begin{enumerate}
  \item $|H(z)|=O\left(\exp\left(\f{C}{1-|z|^\eta}\right)\right)$ for some $C>0$ and $0<\eta<1$ as $|z|\to 1^-$, and
\item $\f{H(e^{-t+iu})}{H(e^{-t})}\to 1$ uniformly in $|u|\leq C' t^{1+\eta'}$ as $t\to 0^+$ for some $C'>0$ and $0<\eta'<\f{1-\eta}{2}$.
  \end{enumerate}
  Then the $n^{th}$ coefficient of the generating function $P(x)H(x)$ equals
  \[
p(n)H(e^{-\pi/\sqrt{6n}})(1+o(1))+O(\exp(-C''n^{1/2-\eta'}))
  \]
  for some constant $C''>0$.\footnote{The dependence of the $o(1)$ term on the estimates in (1) and (2) is not made explicit in the statement of the result in~\cite{GrabnerKnopfmacherWagner2014}, though a short inspection of the proof shows that the dependence is good enough to give us bounds of $O_k(\f{1}{\log{n}})$ for the $o_k(1)$ terms in the conclusion of Lemma~\ref{lem3.1}. This would ultimately also lead to a bound of $O(\f{1}{\log{n}})$ in Theorem~\ref{thm1.1}. Since we get a stronger bound in Theorem~\ref{thm1.2}, we will not put any effort into obtaining explicit bounds for the $o_k(1)$ terms in Lemma~\ref{lem3.1}.}
\end{theorem}

It thus remains to verify that $F_k$ and $G_k$ satisfy the hypotheses of Theorem~\ref{thm3.2} and estimate $F_k(e^{-\pi/\sqrt{6n}})$ and $G_k(e^{-\pi/\sqrt{6n}})$. Noting that $F_k(e^{-v})=\sum_{j=0}^\infty\f{k2^je^{-vk2^j}}{1-e^{-vk2^j}}$ converges whenever $\RE{v}>0$ and using basic properties of the Mellin transform (as can be found in~\cite{FlajoletGourdonDumas1995}), we see that $F_k(e^{-t})$ has Mellin transform equal to
\[
\mathcal{M}[F_k](s):=\f{k^{-(s-1)}}{1-2^{-(s-1)}}\Gamma(s)\zeta(s),
\]
and thus equals
\begin{equation}\label{eq3.1}
F_k(e^{-t})=\f{1}{2\pi i}\int_{2-i\infty}^{2+i\infty}\f{k^{-(s-1)}}{1-2^{-(s-1)}}\Gamma(s)\zeta(s)t^{-s}ds
\end{equation}
whenever $t>0$ by Mellin inversion. The function $\mathcal{M}[F_k](s)t^{-s}$ has a double pole at $s=1$, simple poles at $s=0,-1,-3,-5,\dots$, and simple poles at $s=1+\f{2\pi i\ell}{\log{2}}$ for $\ell\in\Z\setminus\{0\}$. The pole at $s=1$ has residue
\[
-\f{\log{t}+\log{k}}{t\log{2}}+\f{1}{2t},
\]
the poles at $-m$ for $m\in\Z_{\geq 0}\setminus 2\N$ have residue
\[
\f{(-1)^mk^{m+1}\zeta(-m)}{m!(1-2^{m+1})}t^m,
\]
and the poles at $1+\f{2\pi i\ell}{\log{2}}$ for $\ell\in\Z\setminus\{0\}$ have residue
\[
\f{k^{-\f{2\pi i\ell}{\log{2}}}\Gamma(1+\f{2\pi i\ell}{\log{2}})\zeta(1+\f{2\pi i\ell}{\log{2}})}{\log{2}}t^{-1-\f{2\pi i\ell}{\log{2}}}.
\]
We use the rapid decay of $\Gamma$ in vertical strips to push the contour in~\eqref{eq3.1} all the way to the left to thus deduce that $F_k(e^{-t})$ equals
\[
-\f{\log{t}+\log{k}}{t\log{2}}+\f{1}{2t}+\sum_{m=0}^\infty\f{(-1)^mk^{m+1}\zeta(-m)}{m!(1-2^{m+1})}t^m+\sum_{\ell\neq 0}\f{k^{-\f{2\pi i\ell}{\log{2}}}\Gamma(1+\f{2\pi i\ell}{\log{2}})\zeta(1+\f{2\pi i\ell}{\log{2}})}{t\log{2}}t^{-\f{2\pi i\ell}{\log{2}}},
\]
from which it follows that $F_k(e^{-v})=-\f{\log{v}}{v\log{2}}+O_k(\f{1}{|v|})$ when $|\IM{v}|\leq \RE{v}$ as $\RE{v}\to 0^+$. A similar analysis shows that
$G_k(e^{-t})-F_k(e^{-t})^2$ equals
\[
-\f{\log{t}+\log{k}}{t^2\log{2}}+\f{2+\log{2}}{2t^2\log{2}}+\sum_{m=0}^\infty\f{(-1)^{m}k^{m+2}\zeta(-m-1)}{m!(1-2^{m+2})}t^{m}+\sum_{\ell\neq 0}\f{k^{-\f{2\pi i\ell}{\log{2}}}\Gamma(2+\f{2\pi i\ell}{\log{2}})\zeta(1+\f{2\pi i\ell}{\log{2}})}{t^2\log{2}}t^{-\f{2\pi i\ell}{\log{2}}},
\]
from which it follows that $G_k(e^{-v})=(\f{\log{v}}{v\log{2}})^2+O_k(\f{|\log{v}|}{|v|^2})$ when $|\IM(v)|\leq \RE{v}$ as $\RE{v}\to 0^+$.

Since $|F_k(e^{-t+iu})|\leq F_k(e^{-t})$ and $|G_k(e^{-t+iu})|\leq G_k(e^{-t})$ for all $t>0$ and $u\in[0,2 \pi)$ by our original expressions for $F_k$ and $G_k$, the asymptotic expressions for $F_k(e^{-t})$ and $G_k(e^{-t})$ imply that condition~(1) of Theorem~\ref{thm3.2} is easily satisfied for $H=F_k$ and $H=G_k$ for any $\eta>0$. When $|u|\leq t^{4/3}$ (say) we also get that $\f{F_k(e^{-t+iu})}{F_k(e^{-t})}=1+O_k(\f{1}{|\log{t}|})$ and $\f{G_k(e^{-t+iu})}{G_k(e^{-t})}=1+O_k(\f{1}{|\log{t}|})$, so that condition~(2) of Theorem~\ref{thm3.2} is satisfied as well. The lemma now follows by noting that $F_k(e^{-\pi/\sqrt{6n}})=\f{\sqrt{6}}{2\pi\log{2}}\sqrt{n}\log{n}(1+o_k(1))$ and $G_k(e^{-\pi/\sqrt{6n}})=(\f{\sqrt{6}}{2\pi\log{2}}\sqrt{n}\log{n})^2(1+o_k(1))$.
\end{proof}

Now we can prove Theorem~\ref{thm1.1}.
\begin{proof}[Proof of Theorem~\ref{thm1.1}]
  By Lemma~\ref{lem3.1} and Chebyshev's inequality, we get that
\[
M_\mu^{(1)},M_\mu^{(3)},M_\mu^{(5)}\geq\left(\f{1}{\log{2}}-\f{1}{100}\right)\f{\sqrt{6}}{2\pi}\sqrt{n}\log{n}
\]
for all but a $o(1)$-proportion of partitions $\mu$ of $n$. For such $\mu$, we thus have
  \begin{align*}
\left\lfloor\log_2\f{M_{\mu}^{(k)}}{k}\right\rfloor&\geq \log_2\f{(\f{1}{\log{2}}-\f{1}{100})\f{\sqrt{6}}{2\pi}\sqrt{n}\log{n}}{k}-\left\{\log_2\f{(\f{1}{\log{2}}-\f{1}{100})\f{\sqrt{6}}{2\pi}\sqrt{n}\log{n}}{k}\right\}
 \\
&=\log_2\f{(1+\f{1}{100})\f{\sqrt{6}}{2\pi}\sqrt{n}\log{n}}{k}+\log_2\f{\f{1}{\log{2}}-\f{1}{100}}{1+\f{1}{100}}-\left\{\log_2\f{(\f{1}{\log{2}}-\f{1}{100})\f{\sqrt{6}}{2\pi}\sqrt{n}\log{n}}{k}\right\}
\end{align*}
for $k=1,3,$ and $5$, where $\{r\}$ denotes the fractional part of $r$. Note that $\log_21=0$, $\log_23=1.58\dots$, $\log_2{5}=2.32\dots$, and $\log_2\f{\f{1}{\log{2}}-\f{1}{100}}{1+\f{1}{100}}=.504\dots$. So, for all $n$, we certainly have
\[
\left\{\log_2\f{(\f{1}{\log{2}}-\f{1}{100})\f{\sqrt{6}}{2\pi}\sqrt{n}\log{n}}{k}\right\}\leq \log_2\f{\f{1}{\log{2}}-\f{1}{100}}{1+\f{1}{100}}
\]
for at least one of $k=1,3,$ or $5$, since every $r\in[0,1)$ satisfies $\{r-\log_2{k}\}\leq \log_2\f{\f{1}{\log{2}}-\f{1}{100}}{1+\f{1}{100}}$ for at least one of $k=1,3,$ or $5$. We conclude that $\tilde{\mu}$ has a part of size at least $(1+\f{1}{100})\f{\sqrt{6}}{2\pi}\sqrt{n}\log{n}$ for all but a $o(1)$-proportion of partitions $\mu$, so that Theorem~\ref{thm1.1} follows by Lemmas~\ref{lem2.1},~\ref{lem2.2}, and~\ref{lem2.3}.
\end{proof}

\section{Proof of Theorem~\ref{thm1.2}}\label{sec4}
Let $q_p(n)$ denote the number of partitions of $n$ into powers of $p$ and, for $k_1,\dots,k_M\in\N$, let $r_{k_1,\dots,k_M;p}(n)$ denote the number of partitions of $n$ into parts not of the form $k_ip^j$. To prove Theorem~\ref{thm1.2}, it suffices to bound
\begin{equation}\label{eq4.1}
\f{1}{p(n)}\sum_{\substack{\ell_i\leq (1+\delta_p)\f{\sqrt{6}}{2\pi}\sqrt{n}\log{n} \\ k_i\mid \ell_i \\ i=1,\dots,M}}r_{k_1,\dots,k_M;p}\left(n-\sum_{i=1}^M\ell_i\right)\prod_{i=1}^Mq_p\left(\f{\ell_i}{k_i}\right)
\end{equation}
for some $k_1,\dots,k_M$ all satisfying
\[
\left\{\log_2\f{(1+\delta_p)\f{\sqrt{6}}{2\pi}\sqrt{n}\log{n}}{k_i}\right\}\leq\log\f{1+\delta_p}{1+\f{\delta_p}{2}},
\]
since, like in the proof of Theorem~\ref{thm1.1}, this will guarantee that, outside of a set of partitions $\mu$ of $n$ of density~\eqref{eq4.1}, $\tilde{\mu}$ has a part of size at least $(1+\f{\delta_p}{2})\f{\sqrt{6}}{2\pi}\sqrt{n}\log{n}$. We start by showing that, under certain conditions on $k_1,\dots,k_M$,~\eqref{eq4.1} is approximately the product of the probabilities that $M_{\mu}^{(k_i)}\leq(1+\delta_p)\f{\sqrt{6}}{2\pi}\sqrt{n}\log{n}$.

\begin{lemma}\label{lem4.1}
  Let $M\leq n^{1/4-\ve}$, $k_1,\dots,k_M\leq n^{1/4-\ve}$, and $\gamma>0$. We have
  \begin{align*}
    \f{1}{p(n)}\sum_{\substack{\ell_i\leq \gamma\sqrt{n}\log{n} \\ k_i\mid \ell_i \\ i=1,\dots,M}}&r_{k_1,\dots,k_M;p}\left(n-\sum_{i=1}^M\ell_i\right)\prod_{i=1}^Mq_p\left(\f{\ell_i}{k_i}\right)\\
   &=\prod_{i=1}^M\left(\f{1}{p(n)}\sum_{\substack{\ell_i\leq \gamma\sqrt{n}\log{n} \\ k_i\mid\ell_i}}r_{k_i,p}(n-\ell_i)q_p\left(\f{\ell_i}{k_i}\right)\right)(1+O_p(n^{-\ve}))
  \end{align*}
as $n\to\infty$.
\end{lemma}
\begin{proof}
  We first get an asymptotic for $r(m)=r_{k_1,\dots,k_M;p}(m)$, which has generating function
  \[
P(x)\prod_{i=1}^M\prod_{j=0}^{\infty}(1-x^{k_ip^j})=:P(x)H_{k_1,\dots,k_M;p}(x)=P(x)H(x),
  \]
when $m\geq\f{n}{2}$. As in Section~\ref{sec3}, we can do this using a saddle-point argument, though $H$ is not quite well behaved enough for the results of~\cite{GrabnerKnopfmacherWagner2014} to apply. But doing the saddle-point analysis from scratch does not take very long, and also does not stray far from the arguments in~\cite{GrabnerKnopfmacherWagner2014}. So, setting $t_0=\f{\pi}{\sqrt{6m}}$, we have that $r_{k_1,\dots,k_M;p}(m)$ equals
\begin{align*}
&\int\limits_{|u|\leq t_0^{5/4}}P(e^{-t_0+iu})H(e^{-t_0+iu})e^{-m(-t_0+iu)}du+\int\limits_{t_0^{5/4}<|u|\leq\pi}P(e^{-t_0+iu})H(e^{-t_0+iu})e^{-m(-t_0+iu)}du\\
&=:I_1+I_2.
\end{align*}

We will first deal with $I_1$, which contributes the main term of the asymptotic. Note that we can write
\[
\log{H(e^{-v})}=-\sum_{i=1}^M\sum_{j=0}^{\infty}\sum_{\ell=1}^\infty\f{1}{\ell}e^{-v\ell k_ip^j}
\]
when $\RE{v}>0$. So, $\log{H(e^{-t})}$ has Mellin transform equal to
\[
-\sum_{i=1}^m k_i^{-s}\Gamma(s)\zeta(s+1)(1-p^{-s})\1,
\]
and thus equals
\begin{equation}\label{eq4.2}
-\f{1}{2\pi i}\sum_{i=1}^M\int_{\f{1}{2}-i\infty}^{\f{1}{2}+i\infty}(tk_i)^{-s}\Gamma(s)\zeta(s+1)(1-p^{-s})\1 ds
\end{equation}
using Mellin inversion. The poles of
\[
  (tk_i)^{-s}\Gamma(s)\zeta(s+1)(1-p^{-s})\1
\]
consist of a triple pole at $s=0$ with residue
  \[
\f{(\log{t})^2}{2\log{p}}+\f{2\log{k_i}-\log{p}}{2\log{p}}\log{t}+\f{6(\log{k_i})^2-6\log{k_i}\log{p}+(\log{p})^2+\gamma'}{12\log{p}},
  \]
where $\gamma'$ is an absolute constant, simple poles at $s=-r$ for all $r\in\{1\}\cup 2\N$ with residue
  \[
(-1)^{r}\zeta(1-r)\f{(tk_i)^{r}}{r!(1-p^{r})},
  \]
and simple poles at $s=\f{2\pi i \ell}{\log{p}}$ for all $\ell\in\Z\sm\{0\}$ with residue
  \[
\Gamma\left(\f{2\pi i \ell}{\log{p}}\right)\zeta\left(1+\f{2\pi i \ell}{\log{p}}\right)\f{(tk_i)^{-2\pi i\ell/\log{p}}}{\log{p}}.
  \]
We push the contour in~\eqref{eq4.2} all the way to the left using the rapid decay of $\Gamma$ in vertical strips to deduce that $\log{H(e^{-v})}$ equals
\begin{align*}
-\sum_{j=1}^M\bigg[\f{(\log{v})^2}{2\log{p}}&+\f{2\log{k_j}-\log{p}}{2\log{p}}\log{v}+\f{6(\log{k_j})^2-6\log{k_j}\log{p}+(\log{p})^2+\gamma'}{12\log{p}}\\
  &+\sum_{r=1}^\infty(-1)^{r}\zeta(1-r)\f{(vk_j)^{r}}{r!(1-p^{r})}+\sum_{|r|>0}\Gamma\left(\f{2\pi i r}{\log{p}}\right)\zeta\left(1+\f{2\pi i r}{\log{p}}\right)\f{(vk_j)^{-2\pi ir/\log{p}}}{\log{p}}\bigg]
\end{align*}
when $\RE{v}<n^{-1/4}$ and $|\IM{v}|\leq\RE{v}$. As a consequence, we have
\[
\log{\f{H(e^{-t_0+iu})}{H(e^{-t_0})}}=h(M,m)u+O_p(m^{-\ve})
\]
whenever $|u|\leq t_0^{5/4}$ for some $|h(M,m)|=O_p(M\sqrt{m}\log{m})$. We also have
\[
\log\f{P(e^{-t_0+iu})}{P(e^{-t_0})}=\f{\pi^2}{6}\left(i\f{u}{t_0^2}-\f{u^2}{t_0^3}-i\f{u^3}{t_0^4}+\f{u^4}{t_0^5}\right)+O(m^{-1/8})
\]
whenever $|u|\leq t_0^{5/4}$ (see, for example, Section VIII.6 of~\cite{FlajoletSedgewick2009}). So, using that $\f{\pi^2}{6}\cdot\f{u}{t_0^2}=mu$, we get that
\[
I_1=P(e^{-t_0})H(e^{-t_0})e^{mt_0}(1+O_p(m^{-\ve}))\int\limits_{|u|\leq t_0^{5/4}}\exp\left(\f{\pi^2}{6}\left(-\f{u^2}{t_0^3}-i\f{u^3}{t_0^4}+\f{u^4}{t_0^5}\right)+h(M,m)u\right)du.
\]
Making the change of variables $u\mapsto \f{u}{\f{\pi}{\sqrt{3}}t_0^{-3/2}}$ in the integral above gives that $I_1$ equals $\f{\sqrt{3}}{\pi}P(e^{-t_0})H(e^{-t_0})e^{mt_0}(1+O_p(m^{-\ve}))t_0^{3/2}$ times the quantity
\[
\int\limits_{|u|\leq \f{\pi}{\sqrt{3}}t_0^{-1/4}}\exp\left(\f{\sqrt{3}h(M,m)t_0^{3/2}}{\pi}u-\f{u^2}{2}-\f{3\sqrt{3t_0}}{\pi^3}iu^3+\f{9t_0}{\pi^4}u^4\right)du=:I_1'.
\]
Note that
\[
I_1'=\int\limits_{-t_0^{-\ve}}^{t_0^{-\ve}}\exp\left(\f{\sqrt{3}h(M,m)t_0^{3/2}}{\pi}u-\f{u^2}{2}-\f{3\sqrt{3t_0}}{\pi^3}iu^3+\f{9t_0}{\pi^4}u^4\right)du+O(e^{-\Omega_p(m^{\ve})}),
\]
so that
\[
I_1'=(1+O_p(m^{-\ve/2}))\int\limits_{-t_0^{-\ve}}^{t_0^{-\ve}}\exp\left(-\f{u^2}{2}\right)du.
\]
The Gaussian integral above can be extended to one over all of $\R$ at the cost of an error of $O(e^{-m^{\ve}})$, from which it follows that
\[
I_1=p(m)H(e^{-\pi/\sqrt{6m}})(1+O_p(m^{-\ve/2})).
\]

To bound $I_2$, note that
\[
|H(e^{-t_0+iu})|\leq \prod_{i=1}^M\prod_{j=0}^\infty(1+e^{-t_0k_ip^j})<\prod_{i=1}^M\prod_{j=0}^\infty\f{1}{1-e^{-t_0k_ip^j}}=\f{1}{H(e^{-t_0})}\leq\exp(O_p(M(\log{m})^2))
\]
for all $u\in[0,2\pi)$. We combine this with the bound
\[
\f{P(e^{-t+iu})}{P(e^{-t})}\leq\exp\left(-\f{1}{t(1+(\f{\pi t}{2u})^2)}+O(t)\right)
\]
from Lemma~3.1 of~\cite{GrabnerKnopfmacherWagner2014}, which is valid for all $|u|\leq \pi$ as $t\to 0^+$, to get that
\[
I_2\leq P(e^{-t_0})\exp(t_0m-m^{1/4}+O_p(m^{\f{1}{4}-\ve}(\log{m})^2)),
\]
which is at most $p(m)\exp(-\Omega_p(m^{1/4}))$ for $m$ sufficiently large by the standard estimate for $p(m)$. We thus conclude that
\begin{equation}\label{eq4.3}
r_{k_1,\dots,k_M;p}(m)=p(m)H(e^{-\pi/\sqrt{6m}})(1+O_p(m^{-\ve/2})).
\end{equation}

Now, to finish, we use that
\[
\f{p\left(n-\sum_{i=1}^M\ell_i\right)}{p(n)}=(1+O(n^{-1/4}))\prod_{i=1}^M\f{p(n-\ell_i)}{p(n)}
\]
and
\[
H_{k_1,\dots,k_M;p}(e^{-\pi/\sqrt{6(n-\sum_{i=1}^M\ell_i)}})=(1+O_p(n^{-2\ve}(\log{n})^2))\prod_{i=1}^MH_{k_i;p}(e^{-\pi/\sqrt{6(n-\ell_i)}})
\]
whenever $\ell_1,\dots,\ell_M\leq\gamma\sqrt{n}\log{n}$ to get
\[
\f{r_{k_1,\dots,k_M;p}(n-\sum_{i=1}^M\ell_i)}{p(n)}=(1+O_p(n^{-2\ve}(\log{n})^2))\prod_{i=1}^M\f{r_{k_i;p}(n-\ell_i)}{p(n)},
\]
from which the conclusion of the lemma follows.
\end{proof}

To finish the proof of Theorem~\ref{thm1.2}, we will also need a result of Mahler~\cite{Mahler1940} that counts the number of partitions of an integer into powers of any fixed positive integer.
\begin{lemma}[Mahler,~\cite{Mahler1940}]\label{lem4.2}
The number of partitions of $n$ into powers of $p$ equals
\[
\exp\left(\f{1}{2\log{p}}\left(\log\f{n/p}{\log{n/p}}\right)^2+\left(\f{1}{2}+\f{1}{\log{p}}+\f{\log\log{p}}{\log{p}}\right)\log{n}+O_p(\log\log{p})\right).
\]
\end{lemma}

Now we can prove Theorem~\ref{thm1.2}.
\begin{proof}[Proof of Theorem~\ref{thm1.2}]

We show that there exist $\delta_p,\ve_p,\gamma_p'>0$ such that, for any collection of $M\geq Cn^{1/4-\ve_p}$ distinct $k_1,\dots,k_M$ between $\f{1}{p^2}n^{1/4-\ve_p}$ and $n^{1/4-\ve_p}$, there exists some $i$ such that $M_\mu^{(k_i)}\geq(1+\delta_p)\f{\sqrt{6}}{2\pi}\sqrt{n}\log{n}$ for all partitions $\mu$ of $n$ outside of a set of density $\exp(-\Omega_{p,C}(n^{\gamma'_p}))$ for some $i=1,\dots,M$. That there are $\Omega_p(n^{1/4-\ve_p})$ many $k_1,\dots,k_M$ in the interval $[\f{1}{p^2}n^{1/4-\ve_p},n^{1/4-\ve_p}]$ that all satisfy
\[
\left\{\log_p\f{(1+\delta_p)\f{\sqrt{6}}{2\pi}\sqrt{n}\log{n}}{k}\right\}\leq\log_p\f{1+\delta_p}{1+\f{\delta_p}{2}}
\]
(once $\delta_p$ has been fixed) is an immediate consequence of the fact that $\{\log_p{k}\}$ has distribution function $\f{p^t-1}{p-1}$ in intervals of the form $[p^a,p^{a+1}]$. As in the proof of Theorem~\ref{thm1.1}, these two results together imply that $\tilde{\mu}$ has a part of size at least $(1+\f{\delta_p}{2})\f{\sqrt{6}}{2\pi}\sqrt{n}\log{n}$ with probability at least $1-\exp(-\Omega_p(n^{\gamma_p}))$, from which Theorem~\ref{thm1.2} will immediately follow using Lemmas~\ref{lem2.1},~\ref{lem2.2}, and~\ref{lem2.3}.

Set
\[
f_p(k):=\f{1}{p(n)}\sum_{\substack{\ell\leq \gamma\sqrt{n}\log{n} \\ k\mid \ell}}r_{k;p}(n-\ell)q_p\left(\f{\ell}{k}\right)
\]
for every $k\leq n^{1/4-\ve}$. By~\eqref{eq4.3}, Lemma~\ref{lem4.2}, and the standard estimate for $p(n)$, the proportion $\f{r_{k;p}(n-\ell)q_p\left(\f{\ell}{k}\right)}{p(n)}$ equals a quantity that's $\exp(O_p(\log\log{n}))$ times
\[
\exp\left(-\f{\log(t_0k)^2}{2\log{p}}+\f{\log(t_0k)}{2}+\f{\log(\f{\ell/pk}{\log(\ell/pk)})^2}{2\log{p}}+\left(\f{1}{2}+\f{1}{\log{p}}+\f{\log\log{p}}{\log{p}}\right)\log\f{\ell}{k}-\f{\pi \ell}{\sqrt{6n}}\right),
\]
where $t_0=\f{\pi}{\sqrt{6(n-\ell)}}$. When $p=2$ and $\ell\leq\f{1}{10000}\sqrt{n}\log{n}$, the above is $O(\f{1}{n})$, and when $p$ is an arbitrary prime, $\f{1}{p^2}n^{1/4-\ve}\leq k \leq n^{1/4-\ve}$, and $\ell=\gamma\f{\sqrt{6}}{2\pi}\sqrt{n}\log{n}$, the above equals
\[
\exp\left(\left(\f{(\f{1}{4}+\ve)(1+\log(\f{2\gamma}{p(1+4\ve)})+\log\log{p})}{\log{p}}-\f{\gamma}{2}\right)\log{n}+O_p(\log\log{n})\right)
\]

Now consider the function
\[
g_p(\gamma,\ve):=\f{(\f{1}{4}+\ve)(1+\log(\f{2\gamma}{p(1+4\ve)})+\log\log{p})}{\log{p}}-\f{\gamma}{2}.
\]
It follows from a small amount of calculus that when $p=2$ and $\ve_2=\f{(2+2\delta)\log{2}-1}{4}$, we have
\[
g_2(\gamma,\ve_2)<-\f{1}{4}-\ve_2+\f{1}{4}\left(\delta+(2+2\delta)\log\left(1-\f{2\delta}{4+4\delta}\right)\right)
\]
 whenever $\gamma<1+\f{\delta}{2}$. By fixing $\delta>0$ sufficiently small and using that $\log\left(1-\f{2\delta}{4+4\delta}\right)<-\f{2\delta}{4+4\delta}$, we thus get that $f_2(k)=O(\f{(\log{n})^{O(1)}}{n^{\Omega(1)}})$ whenever $\gamma<1+\f{\delta}{2}$ and $\f{1}{4}n^{1/4-\ve_2}\leq k\leq n^{1/4-\ve_2}$. This implies that there exists a $\gamma_2'$ such that, for any $\f{1}{4}n^{1/4-\ve_2}\leq k_1,\dots,k_M\leq n^{1/4-\ve_2}$ with $M\geq Cn^{1/4-\ve_2}$, we have
\[
\prod_{i=1}^Mf_2(k_i)=O(\exp(-\Omega_C(n^{\gamma'_2}))),
\]
from which the desired result for $p=2$ now follows from Lemma~\ref{lem4.1}.

When $p>2$, we have
\[
f_p(k)=1-\f{1}{p(n)}\sum_{\substack{\ell>\gamma\sqrt{n}\log{n} \\ k\mid \ell}}r_{k;p}(n-\ell)q_p\left(\f{\ell}{k}\right),
\]
so that
\[
f_p(k)\leq 1-\f{\gamma\beta\sqrt{n}\log{n}}{kp(n)}\min_{\substack{\gamma\sqrt{n}\log{n}<\ell\leq \gamma(1+\beta) \\ k\mid \ell}}r_{k;p}(n-\ell)q_p\left(\f{\ell}{k}\right)
\]
for any $\beta>0$. When $\delta>0$, it follows from some more calculus that $g_p(1+\delta,\ve)$ is decreasing as $\ve$ increases and attains a maximum of $-\f{1}{2}+\f{1-2\delta\log{p}-\log{\f{p}{2-2\delta}}+\log\log{p}}{4\log{p}}$ at $\ve=0$ as $\ve$ ranges over $[0,\f{1}{4}]$. When $p>13$, the quantity $\f{1-2\delta\log{p}-\log{\f{p}{2-2\delta}}+\log\log{p}}{4\log{p}}$ is negative for all $\delta>0$, but when $p\leq 13$ it is positive for $\delta>0$ sufficiently small. As a consequence, there exist $\delta_p,\ve_p,\alpha_p>0$ satisfying $\alpha_p>\ve_p$ such that $f_p(k)\leq 1-\Omega_p(n^{1/4-\alpha_p})$ whenever $\f{1}{p^2}n^{1/4-\ve_p}\leq k\leq n^{1/4-\ve_p}$. Thus, for any $\f{1}{p^2}n^{1/4-\ve_2}\leq k_1,\dots,k_M\leq n^{1/4-\ve_2}$ with $M\geq Cn^{1/4-\ve_p}$, we have
\[
\prod_{i=1}^Mf_p(k_i)=O_p(\exp(-\Omega_{p,C}(n^{\gamma'_p})))
\]
for $\gamma_p':=\alpha_p-\ve_p$, from which the desired result for $2<p\leq 13$ now follows from Lemma~\ref{lem4.1}.
\end{proof}

\bibliographystyle{plain}
\bibliography{bib}

\end{document}